\documentclass[preprint,12pt]{elsarticle}

\usepackage{amssymb}
\usepackage{amsthm}
\usepackage{amsmath}
\usepackage{mathrsfs}

\newtheorem{thm}{Theorem}[section]

 \newtheorem{prop}[thm]{Proposition}
 \newdefinition{defn}[thm]{Definition}
 \newdefinition{exam}[thm]{Example}
 \theoremstyle{remark}
 \newtheorem{rem}[thm]{Remark}
\newcommand{\abs}[1]{\left\vert#1\right\vert}

 \newcommand{\norm}[1]{\left\Vert#1\right\Vert}

\begin{document}

\begin{frontmatter}






\title{Coarse embedding into uniformly convex Banach space}
\author{Qinggang Ren\fnref{fn1}}
\address{Department of Mathemetics, Kyoto University, Kyoto, Japan, 606-8502.}
\ead{qinggang.ren@hw4.ecs.kyoto-u.ac.jp} \fntext[fn1]{The author was
supported by GCOE pogramm.} \maketitle
\begin{abstract}

In this paper, we study the coarse embedding into Banach space. We
proved that under certain conditions, the property of embedding into
Banach space can be preserved under taking the union the metric
spaces. For a group $G$ strongly relative hyperbolic to a subgroup
$H$, we proved that if $H$ admits a coarse embedding into a
uniformly convex Banach space,
 so is $B(n)=\{g\in G|\abs{g}_{S\cup\mathscr{H}}\leq n\}$.
\end{abstract}

\end{frontmatter}


\section{Introduction}
  After Gromov pointed out that the coarse embedding(also refer as uniform embedding) should be relevant to Novikov conjecture\cite{Gro}\cite{fr},
G.Yu introduced a property called property A for discrete metric
spaces\cite{Yu}. A metric space with property A admits a coarse
embedding into a Hilbert space. And G.Yu proved the coarse
Baum-Connes conjecture holds for the metric spaces with bounded
geometry, which admits a coarse embedding into a Hilbert
space\cite{Yu}. Subsequently G.Kasparov and G.Yu proved the coarse
geometric Novikov conjecture holds for discrete metric spaces with
bounded geometry, which admits a coarse embedding into a uniformly
convex Banach space\cite{KY}. Coarse embedding into Hilbert space
has been studied deeply these years, see \cite{DG}. But there are
less results on the coarse embedding into uniformly convex Banach
space. Also, V.Lafforgue constructed an example which can not be
coarse embedded into uniformly convex Banach spaces\cite{laf}. We
should mention that N. Brown and E. Guentner proved that every
metric space with bounded geometry admits a coarse embedding into a
strictly convex and reflexive Banach space\cite{BE}. \\
  \indent In this paper, we study the coarse embedding into a uniformly convex
Banach space. We first rewrite the condition for coarse embedding
into a uniformly Banach space.
\begin{thm}
Let X be a metric space and $E$ be a Banach space, and $1 \leq
p<+\infty$. If there is a $\delta
>0$, such that for every $R>0 , \varepsilon >0$, there is a map
$\varphi:X\rightarrow  E $ satisfying:
\\(1) $\sup \{ \norm{\varphi (x)- \varphi (y)
} : x,y \in X,\ d(x,y) \leq R \} \leq \varepsilon$.
\\(2) $\forall m \in \mathbb
N,\  \sup \{ \norm{\varphi (x)- \varphi (y) }: x,y \in X, \ d(x,y)
\leq m \}< +\infty$.
\\(3) $\lim_{s \rightarrow +\infty}\inf\{\norm{\varphi
(x)- \varphi (y) }:x,y \in X,d(x,y)\geq s \} \geq \delta$.
\\ then X admits a coarse embedding into $E^p$.
\end{thm}
This is generalized from the conditions for coarse embedding into a
Hilbert space\cite{DG}. Using this condition, we study the coarse
embedding under gluing property. This is easy to prove in the case
of Hilbert space, but difficult in the case of Banach space. We only
obtained some partial results.
\begin{prop}
  Let $X$ be a metric space, and $X=X_1 \cup X_2$, with $X_1,X_2$
  admit coarse embedding into a Banach space $E$ and $1 \leq p<+\infty$. If for any $s>0$, there is
   a bounded set $C_s$ such that the sets $\{X_i\setminus C_s\}$ is
  $s$-separated, then $X$ admits a coarse embedding into $E^p$.
\end{prop}
Recall that two subsets $X_1, X_2$ of a metric space $X$ are
s-separated if $d(X_1, X_2)=\inf\{d(x,y),x \in X_1, y \in X_2\} \geq
s$.
\begin{prop}
If $X$ is long range disconnected at infinity and all $\{X^n_i\}$
are equivalently coarse embedded into a Banach space $E$ by coarse
maps $\{\varphi^n_i\}$, then $X$ admits a coarse embedding into
$E^p$.
\end{prop}
\indent In case of infinite union, we have
\begin{prop}
Let $X$ be a metric space with $X=\cup_{i\in I}X_i$, if for any
$s>0$, there is a bounded set $C_s$ with $X_i\cap C_s\ne \emptyset$
for any i and the sets $\{X_i\setminus C_s\}$ is $s$-separated. If
$\{X_i\}$ can be equally coarse embedded into $E$, then $X$ can be
coarse embedded into $E^p$.
\end{prop}

Further, we study the coarse embeddability of the relative
hyperbolic group and prove that:
\begin{thm}
If the group $G$ is strongly relative hyperbolic to a subgroup $H$
and $H$ admits a coarse embedding into a uniformly convex Banach
space$E$, let $B(n)=\{g\in G|\abs{g}_{S\cup\mathscr{H}}\leq n\}$,
then $B(n)$ admits a coarse embedding into $E^p$.
\end{thm}
 The author would like to show his thanks to Professor X.Chen and
Professor G.Yu for helpful discussion.
\section {Coarse geometry and convex Banach space}
We first recall some definitions in coarse geometry\cite{roe}.
\begin{defn} Let $
X,Y $ be  metric
spaces, and $f$ be a map from $X$ to $Y$: \\
(1) The map $f$ is $proper$ if the inverse image, under $f$, of any
bounded subset of $Y$, is a bounded subset of $X$. \\
(2) The map is $bornologous$ if for every $R>0$ there is an $S>0$,
such
that $d(x,y)\leq R$ implies $d(f(x),f(y))\leq S$. \\
(3)$f$ is \emph{coarse} if it is proper and bornologous.\\
\end{defn}
 \indent And we
say that $X$ admits a  coarse embedding into $Y$ if there is a
coarse map $f:X\rightarrow Y$. We usually consider the case where $Y$ is a Banach space.\\
 We will often need to deal with a family of metric spaces.
\begin{defn}A family of metric space $\{X_i\}_{i\in I}$ is called
equivalently coarse embedded into a metric space $Y$ if there exist $\{f_i:X_i\rightarrow Y\}_{i\in I}$ satisfying:\\
  (1)$\forall s\geq 0$, there exist $S\geq 0$, if $d(x_i,x_i')\leq
 s$ then $d(f_i(x_i),f_i(x_i'))\leq S$ for all $i\in I$.\\
  (2)$\forall r\geq 0$, there exist $R\geq 0$, if $d(f_i(x_i),f_i(x_i'))\leq r$ then $d(x_i,x_i')\leq
 R$ for all $i\in I$.
\end{defn}
Uniformly convex Banach space is an important object to study in
classical Banach space theory\cite{JL}.
\begin{defn}A Banach space $E$ is called uniformly convex if for any
$\epsilon>0$, there exists a $\delta>0$, for any $x,y\in E$ with
$\norm{x}=\norm{y}=1$ and $\norm{x-y}\geq \epsilon$, then
$\norm{\frac{x+y}{2}}\leq 1-\delta$.
\end{defn}
We know $\ell^p(1<p<+\infty)$ is uniformly convex Banach space. If
$E$ is a uniformly convex Banach space, let $$E^p=\{x=(x_i)_{i\in
\mathbb N}|x_i \in E(i\in \mathbb N), \sum_{n\in \mathbb
N}\norm{x_i}^p<+\infty\}$$ with the norm $\norm{x}=(\sum_{n\in
\mathbb N}\norm{x_i}^p)^\frac{1}{p}$. If $1<p<+\infty$, $E^p$ is
also a uniformly convex Banach space.
\section{Coarse embedding into uniformly convex Banach space}
We first rewrite the condition for coarse embedding into a uniformly
convex Banach space.
\begin{thm} \label{mainth}
Let X be a metric space and $E$ be a Banach space, and $1 \leq
p<+\infty$. If there is a $\delta
>0$ such that for every $R>0 , \varepsilon >0$ there is a map
$\varphi:X\rightarrow  E $ satisfying:
\\(1) $\sup \{ \norm{\varphi (x)- \varphi (y)
} : x,y \in X,\ d(x,y) \leq R \} \leq \varepsilon$.
\\(2) $\forall m \in \mathbb
N,\  \sup \{ \norm{\varphi (x)- \varphi (y) }: x,y \in X, \ d(x,y)
\leq m \}< +\infty$.
\\(3) $\lim_{s \rightarrow +\infty}\inf\{\norm{\varphi
(x)- \varphi (y) }:x,y \in X,d(x,y)\geq s \} \geq \delta$.
\\ then X admits a coarse embedding into $E^p$.
\end{thm}
\begin{proof} For $n\in \mathbb{N}$, let $R_n=n$, $\varepsilon = \frac {1}{2^n}$, there is a
$\varphi_n :X \rightarrow E$ satisfying the above conditions. And we
can find an $ s_n$ such that $\norm{\varphi _n (x)- \varphi _n (y)}
>\frac {\delta}{2}$ if $d(x,y) \geq s_n$, we can choose$ \{ s_n \}$
to be an increasing sequence. Fix a point $x_0 \in X$ and define
\[
\begin{split}
\varphi :&X\rightarrow E^p\\
   &x \mapsto \bigoplus_{n=1}^\infty(\varphi
_n(x)-\varphi _n(x_0))
\end{split}
\]
\\ it is easy to see that $\norm{\varphi (x)
}<+\infty$ for each $x\in X$. We show that $\varphi $ is coarse.
\\(1) For any $x,y \in X$, assume $k-1<d(x,y) \leq k$, then
\begin{eqnarray*}
\norm{\varphi (x)- \varphi (y)}^p & = & \sum_{n=1}^{+\infty }\norm{
\varphi_n(x)- \varphi_n(y)
}^p\\
 & = & \sum_{n=1}^{k-1}\norm{\varphi_n(x)- \varphi_n(y)
}^p+\sum_{n=k}^{+\infty }\norm{\varphi_n(x)- \varphi_n(y)
}^p\\
 & \leq & \sum_{n=1}^{k-1}\norm{\varphi_n(x)- \varphi_n(y)
}^p+ \sum_{n=k}^{+\infty}\frac {1}{2^{np}}.
\end{eqnarray*}
\\ Let $ C_n^k=sup \{ \norm{\varphi_n(x)-
\varphi_n(y)},d(x,y)\leq k \} $, then  $$\norm{\varphi (x)- \varphi
(y)}^p \leq \sum_{n=1}^{k-1}(C_n^k)^p+1.$$
\\ (2) For any $x,y \in X$, assume $s_{k-1} \leq d(x,y) \leq s_k$, then
\\ $$\norm{\varphi (x)- \varphi (y)}^p=
\sum_{n=1}^\infty \norm{\varphi_n (x)- \varphi_n(y)}^p
\\ \geq \sum_{n=1}^{k-1} \norm{\varphi_n (x)- \varphi_n (y)}
^p >(k-1)(\frac {\delta}{2})^p$$
\\ and $d(x,y) \rightarrow +\infty$ implies $k\rightarrow +\infty$, so $(k-1)(\frac {\delta}{2})^p \rightarrow
+\infty$.
\end{proof}
\begin{exam}
$\ell^p$ ($1\leq p<+ \infty$) satisfies the above conditions for
$E=\ell^p$.
\end{exam}
\begin{proof} Let $\delta =1$, $\forall R>0,\varepsilon >0$, there is a
natural number $h$ such that $\frac{R}{h}<\varepsilon$, define
$\varphi :\ell^p \rightarrow \ell^p$ by $\varphi(x)=\frac{x}{h}$
then
\\(1) $\sup\{ \norm{\varphi (x)- \varphi (y)}, d(x,y) \leq R
\}$=$sup \{ \frac {1}{h}\norm{ x-y}, \ d(x,y)\leq R\}< \varepsilon$.
\\(2) $\sup\{\norm{\varphi (x)- \varphi (y)}, d(x,y) \leq m\}=\frac {m}{h} < +\infty$.
\\(3) $\inf \{\norm{\varphi (x)- \varphi (y)}, d(x,y)
\geq s\} = \frac{s}{h}$, $lim_{s \rightarrow +\infty} \frac{s}{h}=
+\infty$.
\end{proof}
W.B.Johnson and N.L.Randrianarivony proved that $\ell^p(p>2)$ does
not admit a coarse embedding into a Hilbert space\cite{JR}. So the
conditions for coarse embedding into a uniformly convex Banach space
$E$ in the above theorem is very different from the coarse embedding
into Hilbert space\cite{DG}.
\begin{exam}  \label{nccond}
If $X$ admits a coarse embedding into $\ell^p$, then $X$ satisfies
the above conditions.
\end{exam}
\begin{proof}There is a $\psi :X \rightarrow \ell^p$ and $
\rho _{\pm}$, such that
\\ (1)  $\rho_{-}(d(x,y)) \leq \norm{ \psi (x)- \psi (y)}
_p \leq \rho_{+}(d(x,y))$,
\\ (2) $\lim_{r \rightarrow +\infty} \rho_+(r)=+\infty$.
\\ $\forall R>0,\varepsilon >0$, there is a nature number
h such that $\frac{\rho_+(R)}{h} < \varepsilon$
\\ define $\varphi :X\rightarrow \ell^p$ by $\varphi (x)=\frac
{\psi(x)}{h}$ then:
\\ (1)$\sup\{\norm{\varphi (x)- \varphi
(y)}_p, d(x,y)\leq R\}=sup\{\frac{\norm{\psi (x)- \psi (y)}_p}{h},
d(x,y)\leq R\}\leq\varepsilon$.
\\ (2)$\sup\{\norm{\varphi (x)- \varphi (y)}_p,d(x,y)\leq
m\} =sup\{\frac{\norm{\psi (x)- \psi (y)}_p}{h}, d(x,y)\leq m\} \leq
\frac{\rho_+(m)}{h}$.
\\ (3)$\inf\{\norm{\varphi (x)- \varphi (y)}_p,
d(x,y)\geq s\}=inf\{\frac{\norm{\psi (x)- \psi (y)}_p}{h},
d(x,y)\geq s\}\geq \frac{\rho_-(s)}{h}$
\\ and $lim_{s\rightarrow +\infty }\frac{\rho_-(s)}{h}=+\infty$.
\end{proof}

\begin{rem}(a) We can see from the proof that this $\delta$ is not
important. We can take it to be infinity in general, i.e, replace
third condition with $$\lim_{s \rightarrow +\infty}\inf\{\norm{
\varphi (x)- \varphi (y)}_p :x,y \in X,d(x,y)\geq s \} =+\infty.$$
\\
(b) If we take $E=\ell^p(1<p<+\infty)$ and change the condition (2)
of Theorem \ref{mainth} with
 $$\sup \{ \norm{\varphi (x)- \varphi (y)}: x,y \in X,{ } d(x,y) \leq m \}<L_0,\ \forall m \in \mathbb N$$
 for some fixed $ L_0$, by the Mazur map mentioned in ~\cite{BL}, then it can be embedded
into $\ell^2$, a Hilbert space.
\end{rem}

\section{On the union of metric spaces} \label{union}
In this section, we study the coarse embeddability  under taking the
union of metric spaces.
\begin{prop}\label{unionthm}
Let $X$ be a metric space, and $X=X_1 \cup X_2$, with $X_1,X_2$
  admit coarse embedding into a Banach space $E$ and $1 \leq p<+\infty$. If for any $s>0$, there is
   a bounded set $C_s$ such that the sets $\{X_i\setminus C_s\}$ is
  $s$-separated, then $X$ admits a coarse embedding into $E^p$.
\end{prop}
\begin{proof}We first assume that $X_1 \cap X_2 \ne \emptyset$, take
an $x_0 \in X_1 \cap X_2$, and replace $C_s$ with $C_s \cup x_0$ if
necessary, we can assume that $x_0 \in C_s$ for any $s$.
\\ \indent For any $R\geq0$, $\varepsilon\geq0$, there is a bounded set $C_{2R}$ such that
$X_i\setminus C_{2R}$ is 2R-separated. Suppose $C_{2R} \subset
B(x_0,k)$ for some $k$. For $X_i$ admits coarse embedding into $E$,
we can find a number $r>2R+k$ and a map $\varphi_r^i$, such that
\\(1) $\sup \{ \norm{\varphi_r^i (x)- \varphi_r^i (y)
} : x,y \in M_i,{ } d(x,y) \leq r\} \leq \varepsilon$.
\\(2) $\forall m \in \mathbb N,\ \sup \{ \norm{\varphi_r^i (x)- \varphi_r^i (y)
} : x,y \in M_i,{ } d(x,y) \leq m \}< + \infty$.
\\{ }(3) $\lim_{s \rightarrow +\infty}\inf\{\norm{\varphi_r^i
(x)- \varphi_r^i (y)} :x,y \in M_i,d(x,y)\geq S \}=+\infty$.\\
\\
 And we define
\begin{eqnarray*}
  \varphi :&&X\rightarrow (E\oplus E)_p\\
                     &&x \mapsto (\varphi_r^1(a)-\varphi_r^1(x_0),0)\ if\  x \in M_1 \setminus C_{2R},\\
                     &&y \mapsto (0,\varphi_r^2(b)-\varphi_r^2(x_0))\  if\  y \in M_2 \setminus C_{2R},\\
                     &&z \mapsto (0,0) \ if \ z\in C_{2R}
\end{eqnarray*}
We need to verify the conditions in Theorem \ref{mainth}:\\ \indent
(i) For $d(x,y)\leq R$, if $x\in M_1\setminus C_{2R},y\in C_{2R}$,
then $d(x_0,y)\leq k$, we have
 $d(x,x_0)\leq d(x,y)+d(y,x_0)\leq k+R$, so
 $$\norm{\varphi (x)-\varphi
 (y)}=\norm{\varphi_r^1(x)-\varphi_r^1(x_0)}<\epsilon.$$
 \indent If $x\in M_2\setminus C_{2R},y\in C_{2R}$ or  $x,y \in
 C_{2R}$ or $x,y\in M_i\setminus C_{2R}$ for same i, it is similar to prove $\norm{\varphi (x)-\varphi
 (y)}<\epsilon$.\\
(ii) $\forall m>0$, there is a bounded set $C_{m}$, such that
$X_i\setminus C_{m}$ is m-separated, and we can find a number h such
that $C_m \subset B(C_{2R},h)$. For $d(x,y)<m$,\\ if $x \in
X_i\setminus C_{2R},y \in X_i\setminus C_{2R}\ for\ same\ i$ then
$$\norm{\varphi (x)-\varphi
(y)}=\norm{\varphi_r^i(x)-\varphi_r^i(y)}.$$ if $x \in X_i\setminus
C_{2R},y \in C_{2R}$, then $d(x,x_0)\le d(x,y)+d(y,x_0)\le m+k$. And
$$\norm{\varphi (x)-\varphi
(y)}=\norm{\varphi_r^i(x)-\varphi_r^i(x_0)}.$$ if $x\in X_1\setminus
C_{2R},y\in M_2\setminus C_{2R}$, for $d(x,y)\le m$, then either
$x\in C_m$ or $y\in C_m$. Suppose $x\in
C_m$, so $d(x,x_0)\le h+k;d(y,x_0)\le h+k+m$, then\\
$$\norm{\varphi (x)-\varphi (y)}=(\norm{\varphi_r^1 (x)-\varphi_r^1
(x_0)}^p+\norm{\varphi_r^2 (y)-\varphi_r^2 (x_0)
}^p)^{\frac{1}{P}}.$$ Let $t=h+m+k$, we get\\
\begin{eqnarray*}
\sup\{\norm{\varphi (x)-\varphi (y) }, d(x,y)\le m\}\leq
max\{\sup_{d(x,y)\le t}\norm{\varphi_r^i (x)-\varphi_r^i (y)},\\
\sup_{d(x,x_0)\le t,d(y,x_0)\le t}\{(\norm{\varphi_r^1
(x)-\varphi_r^1 (x_0)} ^p+\norm{\varphi_r^2 (y)-\varphi_r^2
(x_0)}^p)^{\frac{1}{P}} \}\} <+\infty.
\end{eqnarray*}

(iii) Let $d(x,y)=s$, let $s$ tends to infinity, \\
if $x,y\in X_i\setminus C_{2R} $, then$\norm{\varphi_r^i
(x)-\varphi_r^i (y)}\rightarrow +\infty$ by the property
of $\varphi_r^i$.\\
if  $\ x\in X_i \setminus C_{2R},y\in C_{2R}$, then for
$d(y,x_0)<k$, so $d(x,y)\rightarrow +\infty$ \ implies \\$
d(x,x_0)\rightarrow +\infty$, so$$\norm{\varphi (x)-\varphi (y)
}=\norm{\varphi_r^i (x)-\varphi_r^i (x_0)}\rightarrow +\infty.$$ if
$x\in X_1\setminus C_{2R},y\in X_2\setminus C_{2R}$,
$d(x,y)\rightarrow +\infty$, implies either $d(x,x_0)\rightarrow
+\infty$ or $d(y,x_0)\rightarrow +\infty$, thus\ $$\norm{\varphi
(x)-\varphi (y) }^p=(\norm{\varphi_r^1 (x)-\varphi_r^1 (x_0)}
^p+\norm{\varphi_r^2 (y)-\varphi_r^2 (x_0)}
^p)^{\frac{1}{P}}\rightarrow +\infty.$$

if $X_1 \cap X_2=\emptyset$, we can assume that $X_i\cap C_s \ne
\emptyset(\forall s>0)$, take  $x_0\in X_1\cap C_{2R}$,\\
$y_0\in X_2\cap C_{2R}$, and define
\begin{equation*}
\begin{split}
 \qquad  \varphi :&X\rightarrow E \oplus E \\
                     &x \mapsto (\varphi_r^1(a)-\varphi_r^1(x_0),0)\  if\  x \in X_1 \setminus C_{2R},\\
                     &y \mapsto (0,\varphi_r^2(b)-\varphi_r^2(y_0))\  if\  y \in X_2 \setminus C_{2R},\\
                     &z \mapsto (0,0) \ if \ z\in C_{2R}
\end{split}
\end{equation*}
The proof follows. Applying the theorem \ref{mainth}, we finish the
proof.
\end{proof}
Gromov introduces the following property for metric space\cite{Gro}.
\begin{defn}
A matric space $X$ is called long range disconnected at infinity if
for every $n\in \mathbb N$, there exist two subsets $X^n_1$ and
$X^n_2$ in $X$ such
that: \\
(1)$d(X^n_1,X^n_2)=inf\{d(x_1,x_2)|x_1\in X^n_1, x_2\in X^n_2\}\geq d$.\\
(2)$X^n_1$ and $X^n_2$ cover almost all $X$, i.e.,
$X\setminus(X^n_1\cup X^n_2)$ is bounded.
\end{defn}
\begin{prop}
If $X$ is long range disconnected at infinity and all $\{X^n_i\}$
are equivalently coarse embedded into a Banach space $E$ by coarse
maps $\{\varphi^n_i\}$, then $X$ admits a coarse embedding into
$E^p$.
\end{prop}
\begin{proof} For any $R\geq 0$, $\epsilon\geq 0$, find an $n\in\mathbb
N$ such that $n>R$. Choose a point $x_n\in X\setminus(X^n_1\cup
X^n_2)$, define
\begin{eqnarray*}
  \varphi :&&X\rightarrow (E\oplus E)_p\\
                     &&x \mapsto (\varphi_1^n(a)-\varphi_1^n(x_n),0)\ if\  x \in X^n_1 ,\\
                     &&y \mapsto (0,\varphi_2^n(b)-\varphi_2^n(x_n))\  if\  y \in X^n_2,\\
                     &&z \mapsto (0,0) \ otherwise.
\end{eqnarray*}
Using the similar argument in proposition \ref{unionthm}, it can be
show that $\varphi$ satisfies the condition of theorem \ref{mainth},
we finish the proof.
\end{proof}
\begin{prop}
Let $X=\cup_{i\in I}X_i$ be a metric space. If for any $s>0$, there
is a bounded set $C_s$ with $X_i\cap C_s\ne \emptyset(\forall i)$
and the sets $\{X_i\setminus C_s\}$ is s-separated. If $X_i$ can be
equally coarse embedded into a Banach space $E$, then $X$ can be
coarse embedded into $E^p$.
\end{prop}
\begin{proof}
$\forall R>0,\epsilon >0$, there is a bounded set $C_R$ such that
$M_i\setminus C_R$ is R-separated, and suppose that $C_R\subset
B(x_0,k)$ for some $k$, take an $r>k+2R$. For $X_i$ is equally
coarse embedded into $E$, we can find $\varphi _r^i:X_i \rightarrow
E$, such that\\
(1) $\underset{i}{\sup}\ \sup\{ \parallel \varphi_r^i
(x)-\varphi_r^i (y)
\parallel<\epsilon,x,y\in X_i, d(x,y)<r\}<\epsilon.$
\\(2)$\underset{i}{\sup}\ \sup\{\parallel \varphi_r^i (x)-\varphi_r^i
(y)\parallel,x,y\in X_i, d(x,y)<m\} < \infty, \forall \
m\in\mathbb{N}.$
\\(3)$\underset{s\rightarrow \infty}{\lim}\ \underset{i}{\inf}\ \inf\{\parallel \varphi_r^i (x)-\varphi_r^i (y)
\parallel,x,y \in X_i, d(x,y)>s\}=\infty.$
\\ For each
$i$, fix an $x_i\in X_i\cap C_R$. And define
\begin{eqnarray*}
\varphi: &M\rightarrow E^p\\
         &a\mapsto \underset{ith\ item}{(0,\dots,\varphi_r^i
         (x)-\varphi_r^i
         (x_i),0.\dots)}\ if \ a\in M_i\setminus C_R\\
         &b\mapsto (0,\dots,0)\ if\ b \in C_R
\end{eqnarray*}
The proof follows using the similar argument in propostion
\ref{unionthm}.
\end{proof}
\section{Relative hyperbolic group}\label{relative}
Let $G$ be a finitely generated group with generating set $S$(closed
under taking inverse) then $G$ is a proper metric space with word
length metric induced by the generating set $S$. Let $H$ be a finite
generated subgroup of $G$. We denoted $\{H\setminus{e}\}$ by
$\mathscr{H}$. Then the Cayley graph $(G,S)$ and
$(G,S\cup\mathscr{H})$ are both metric spaces with word length
metric $d_S,d_{S\cup\mathscr{H}}$, respectively.
\begin{defn}
Let $p$ be a path in $(G,S\cup\mathscr{H})$. An
$\mathscr{H}$-component of $p$ is a maximal sub-path of $p$
contained in a same left coset $gH$. The path is said to be
$without\ backtracking$ if it does not have two distinct
$\mathscr{H}$-component in a same coset $gH_i$.
\end{defn}

\begin{defn}
a path-metric space $X$ is hyperbolic if there exists some
$\delta>0$ such that the $\delta$-neigborhood of any two sides of a
geodesic triangle contain the third side. The group $G$ is said to
be $weakly$ $hyperbolic$ $relative$ to $H$ if the Cayley graph
$(G,S\cup\mathscr{H})$ is hyperbolic.
\end{defn}

\begin{defn}[see\cite{Farb}]
We say the pair $(G,H)$ satisfies the \emph{Bounded Coset
Penetration property(BCP)} if for every $R\geq 0$, there exists
$a=a(R)$ such that if $p,q$ are two geodesics in
($G,S\cup\mathscr{H}$) with $p_-=q_-$ and $d_S(p_+,q_+)\leq R$, \\
(1) Suppose that $p$ has an $\mathscr{H}$-component $s$ with
$d_S(s_-,s_+)\geq a(R)$, then q has an $\mathscr{H}$-component
contains in the same left coset of $s$. \\
(2) Suppose $s,t$ are two $\mathscr{H}$-component of $p,q$
respectively, contained in the same left coset , then
$d_S(s_-,t_-)\leq
a(R)$, $d_S(s_+,t_+)\leq a(R)$.\\
\end{defn}

\begin{defn}
The group $G$ is $\emph{strongly relative hyperbolic}$ to $H$ if it
is weakly hyperbolic to $H$ and satisfies $BCP$.
\end{defn}
Denote by $B(n)=\{g\in G|\abs{g}_{S\cup\mathscr{H}}\leq n\}$. D.Osin
proved in \cite{osin} that $B(n)$ has asymptotic dimension at most
$d$ if the subgroup $H$ have asymptotic dimension at most $d$.
M.Dadarlat and E.Guentner proved $G$ admits a coarse embedding into
a Hilbert space if $H$ admits a coarse embedding into a Hilbert
space \cite{DG2}. We prove that:
\begin{thm}
If $H$ admits a coarse embedding into $E$, then B(n) admits a coarse
embedding into $E^p$ for each $n\in\mathbb N$.
\end{thm}
\begin{proof} we proceed by
induction on n. We have $B(0)=\{e\}$ is trivial. $B(1)=H\cup S$ is
just in the 1-neighborhood of $H$, so can be coarsely embedded. We
assume that $B(n-1)$ is coarsely embedded into $\ell^p$. we know
$$B(n)=(\bigcup_{x\in S}B(n-1)x)\cup B(n-1)H$$ \\ since
$B(n-1)x$ is just in the 1-neighborhood of $B(n-1)$ in $(G,S)$, so
it can be coarsely embedded. We concerned on $B(n-1)H$. we can find
a subset $R(n-1)$ in $B(n-1)$ such that for any $b\in B(n-1)$,
$bH=gH$ for a unique $g\in R(n-1)$. Thus
$$B(n-1)H=\bigsqcup_{g\in R(n-1)}gH.$$

$\forall R\geq 0,\ \epsilon\geq 0$, we have an $a(R)$ from the
$BCP$. We can assume $a(R)\geq R$ and $a(R)$ is increasing.  let
$T_R=\{g\in G|\abs{g}_S\leq a(R)\}$. Let $Y_R=B(n-1)T_R$, then
D.Osin proved that $\{gH\setminus Y_R\}_{g\in R(n-1)}$ is
$R$-separated \cite{osin}. We find maps $\varphi_1$ and $\varphi_2$
for
embedding of $Y_R$ and $H$, respectively, such that\\
\\
(1)$\sup \{ \norm{ \varphi_i(x)- \varphi_i(y) } _p : x,y \in X,{ }
d_S(x,y) \leq 3a(R) \} \leq \varepsilon/2$,
\\(2)$C_m=\sup \{ \norm{\varphi_i(x)- \varphi_i(y)
} _p : x,y \in X,{ } d_S(x,y) \leq m \}< + \infty \quad{ } \forall m
\in \mathbb N$,
\\(3)$\lim_{t \rightarrow +\infty}inf\{\norm{\varphi_i
(x)- \varphi_i (y)} _P :x,y \in X,d_S(x,y)\geq t \} \geq \delta$. \\
\\
We define a map:
$$\varphi:B(n-1)H\rightarrow E\oplus(\bigoplus_{g_i\in
R(n-1)}E)$$ as following.\\
For $x\in g_iH\setminus Y_R$, fix a shortest word $A_i$ for $g_i$ in
$(G,S\cup\mathscr{H})$. Let $A_i=g_i'h_i'$ where $h_i'$ is the
$\mathscr{H}$-component in $g_iH$. Replacing $g_i$ with $g_i'$, we
can assume $g_i$ does not have an $\mathscr{H}$-component in $g_iH$.
Then $x=g_ix_i$ is a geodesic in $(G,S\cup\mathscr{H})$\cite{osin}.
We define $\varphi(x)=\varphi_1(g_i)\oplus\varphi_2(x_i)$ where
$\varphi_2(x_i)$ is in the $g_i$ item in $\oplus_{g_i\in R(n-1)}E)$.
And let $\varphi(y)=\varphi_1(y)$ for $y\in Y_R$. We need to verify
the three conditions of embedding.\\
\\
(1)For $d_S(x,y)\leq R$, we have two cases.\\
\\
 \indent Case a: If $x,y\in Y_R$, $\norm{\varphi_1(x)-\varphi_1(y)}\leq
\epsilon$.\\
\\

 \indent Case b: If $x\in g_iH\setminus Y_R$, we have $y\in g_iH$ by $BCP$.
 If $y\in
 g_iH\setminus Y_R$, let $y=g_iy_i$, $x=g_ix_i$ be the geodesics in
 $(G,S\cup\mathscr{H})$. Thus\\ $d_S(x_i,y_i)=d_S(x,y)\leq R$,
 $$\norm{\varphi(x)-\varphi(y)}=\norm{\varphi_2(x_i)-\varphi_2(y_i)}<\epsilon.$$
  If $y\in g_iH\cap Y_R$, let $y=g_i'y_i'$ be a geodesic in
 $(G,S\cup\mathscr{H})$. For $d_S(x,y)\leq R$, we have $d_S(g_i,g_i')\leq
 a(R)$ and $|y_i|_s\leq a(R)$. Then $$d_S(g_i,y)\leq d_S(g_i,g_i')+d_S(g_i',y)\leq
 2a(R),$$$$|x_i|_s=d_S(g_i,x)\leq d_S(g_i,y)+d_S(y,x)\leq 3a(R).$$
 we have
$\norm{\varphi(x)-\varphi(y)}=(\norm{\varphi(g_i)-\varphi(y)}^p+\norm{\varphi_2(x_i)}^p)^\frac{1}{p}\leq
\epsilon$.\\
\\
\noindent(2)For $m\in \mathbb{N}$ and $d_S(x,y)\leq m$, let
$T_m=\{g||g|_S\leq a(m)\}$ and $Y_m=B(n-1)T_m$. Then
$\{g_iH\setminus Y_m\}$ is $m$-separated. If $x,y$ is both in $Y_R$
or in $g_iH$ for some i, it is easy to see
$\norm{\varphi(x)-\varphi(y)}$ is bounded. So we only need to
consider $x\in g_iH$, $y \in g_jH$ with $i\neq j$. For $d_S(x,y)\leq
m$, either $x$ or $y$ is in $Y_m$. We assume $y\in Y_m$.
 For $x\in
g_iH\cap (Y_m\setminus Y_R)$, let $x=g_ix_i$ be a geodesic in
$(G,S\cup\mathscr{H})$. We have two cases,\\

 \indent Case a: If $y\in g_jH\cap (Y_m\setminus Y_R)$.
 let $y=g_jy_j$ be geodesics in $(G,S\cup\mathscr{H})$.
Then\\
$$\norm{\varphi(x)-\varphi(y)}^p=\norm{\varphi_1(g_i)-\varphi_1(g_j)}^p+\norm{\varphi_2(x_i)}^p+\norm{\varphi_2(y_j)}^p.$$\\
for $d_S(g_i,g_j)\leq a(m)$, $|x|_S\leq a(m)$, $|y_j|_S\leq a(m)$,
we know $\norm{\varphi(x)-\varphi(y)}\leq 3C_m$.\\
\indent Case b: If  $y\in g_jH\cap Y_R$, let $y=g_i'y_i'$ be the
geodesic in $(G,S\cup\mathscr{H})$ with $y_i$ is a
$\mathscr{H}$-component and $|y_i'|_S\leq a(R)$. For $d_S(x,y)$,
then $d_S(g_i,g_i')\leq a(m)$
and $d_S(g_i,y)\leq 2a(m)$, $\abs{x}_S\leq 3a(m)$. Then\\
$$\norm{\varphi(x)-\varphi(y)}^p=\norm{\varphi_1(g_i)-\varphi_1(y)}^p+\norm{\varphi_2(x_i)}^p.$$
we have that $\norm{\varphi(x)-\varphi(y)}\leq 2C_m.$ \\
\\
(3)We have $d_S(x,y)\leq l(x_g)+d_S(g_i,g_j)+l(y_g)$, thus
$d_S(x,y)$ tends infinity implies at least one the three must tends
to infinity. So
$$\lim_{t\rightarrow+\infty}\inf\{\norm{\varphi(x)-\varphi(y)},d_S(x,y)\geq
t\}=\infty.$$ By Theorem \ref{mainth}, $B(n)$ admits a coarse
embedding into $E^p$.
\end{proof}





\bibliographystyle{elsarticle-num}
\bibliography{<your-bib-database>}

\begin{thebibliography}{00}


\bibitem{BL}Y.Benyamini and J.Lindenstrauss, \emph{Geometric nonlinear
functional analysis}, Volume 48 of Colloquium Pubilications.American
Mathematical Society, Providence,R.I.,2000.
\bibitem{BE} N.Brown and E.Guentner
 \emph{Uniform embeddings of bounded geometry spaces into reflexive Banach
space}, Proc. Amer. Math. Soc. 133 (2005), 2045-2050.
\bibitem{DG}M.Dadarlat and E.Guentner, \emph{Constructions preserving
Hilbert space uniform embeddability of discrete groups},
Trans.Amer.Math.Soc.355\\(2003), 3253-3275.
\bibitem{DG2}M.Dadarlat and E.Guentner, \emph{Uniform embeddability of relatively hyperbolic groups}, J.Reine Angew. Math. 612 (2007), 1--15.
\bibitem{Farb}B.Farb, \emph{Relatively hyperbolic groups}. Geom. Funct. Anal. 8 (1998), no. 5, 810--840.
\bibitem{fr}S.Ferry, A.Ranicki and J.Rosenberg(eds.), \emph{Novikov
conjectures, index theorems and rigidity}, London Mathematical
Society Lecture Notes, no. 226, 227, Cambridge University Press,
1995.
\bibitem{Gro}M.Gromov, \emph{Asymptotic invariants of infinite groups, Geometric Group Theory}(A. Niblo and M. Roller,
eds.), London Mathematical Society Lecture Notes, no. 182, Cambridge
University Press, 1993, pp. 1¨C 295.
\bibitem{JR}W.B.Johnson and N.L.Randrianarivony, $\ell_p(p>2)$\emph{does not coarsely embed into a Hilbert space}, Proc. Amer.
Math. Soc. 134 (2006), no. 4, 1045--1050.
\bibitem{JL}W.B.Johnson and J.Lindenstrauss, \emph{Handbook of geometry of Banach spaces}. Amsterdam ; New York : Elsevier, 2001-
\bibitem{KY}G.Kasparov and G.Yu, \emph{The coarse geometric Novikov conjecture and
uniform\ con-vexity}. Adv. Math. 206 (2006), no. 1, 1--56.
\bibitem{laf}V.Lafforgue $Un\ renforcement\ de\ la\ properi\acute{e}t \acute{e}$, Duke Math.J. 143 (2008), no. 3,
559--602.
\bibitem{osin}D.Osin, \emph{Asymptotic dimension of relatively hyperbolic groups}. Int. Math. Res. Not. 2005, no. 35, 2143--2161.
\bibitem{roe}J.Roe, \emph{Lectures on coarse geometry}, University Lecture
Series, vol.31,Americian Mathematical Society, Providence, RI, 2003
\bibitem{Yu}G.Yu, \emph{The coarse Baum-Connes conjecture for
spaces which admit a uniform embedding into Hilbert space}.
Invent.Math.139(2000),no.1,201-240.



\end{thebibliography}



\end{document}